\documentclass[10pt,a4paper,draft]{article}

\usepackage{amsmath,amsthm,amsopn,amsfonts,amssymb,dsfont,color}
\usepackage[dvips]{graphicx}
\usepackage{psfrag}
\usepackage{enumerate}

\usepackage{a4,a4wide}

\def\ep{\varepsilon}
\def\R{\mathbb R}

\newtheorem{theo}{\textbf{Theorem}}[section]

\newtheorem{lem}[theo]{\textbf{Lemma}}
\newtheorem{prop}[theo]{\textbf{Proposition}}

\newtheorem{defi}[theo]{\textbf{Definition}}

\newtheorem{assumption}[theo]{\textbf{Assumption}}
\newtheorem{rem}[theo]{\textbf{Remark}}

\title{Asymptotic analysis of a monostable equation in periodic media}
\date{}

\begin{document}

\maketitle

\begin{center}
{\large\bf Matthieu Alfaro \footnote{ I3M, Universit\'e de
Montpellier 2, CC051, Place Eug\`ene Bataillon, 34095 Montpellier
Cedex 5, France. E-mail: matthieu.alfaro@univ-montp2.fr}
 and  Thomas Giletti
\footnote{IECL, Universit\'{e} de Lorraine, B.P. 70239, 54506
Vandoeuvre-l\`{e}s-Nancy Cedex, France. E-mail:
thomas.giletti@univ-lorraine.fr}.}\\
[2ex]

\end{center}

%\vspace{15pt}

\tableofcontents

\vspace{10pt}

\begin{abstract} We consider a multidimensional monostable reaction-diffusion equation whose nonlinearity
involves periodic heterogeneity. This serves as a model of
invasion for a population facing spatial heterogeneities.
 As a rescaling parameter tends to zero, we prove the
convergence to a limit interface, whose motion is governed by the
minimal speed (in each direction) of the underlying pulsating
fronts. This dependance of the speed on the (moving) normal
direction is in contrast with the homogeneous case and makes the
analysis quite involved. Key ingredients are the recent
improvement \cite{A-Gil} of the well-known spreading
properties \cite{Wein02}, \cite{Ber-Ham-02}, and the solution of a Hamilton-Jacobi equation. \\


\end{abstract}

\section{Introduction}\label{s:intro}

We consider the Cauchy problem
\[
 (P^\ep) \quad\begin{cases}
 \partial _t u^\ep= \ep \Delta u^\ep+\displaystyle \frac 1 \ep f\left(\frac x \ep,u^\ep\right) &\text{in }(0,\infty)\times \R ^N  \vspace{5pt}\\
 u^\ep(0,x)=g(x) &\text{in }\R ^N,
 \end{cases}
\]
where $u$ will typically denotes a population density, and the nonlinearity $f(x,u)$ is periodic in $x\in \R^N$ and of
the monostable type. The parameter $\ep>0$ measures the thickness
of the diffuse interfacial layer, which will account for the invasion front of the population. Our goal is to study the
asymptotic behavior --- or the {\it singular limit}, or the {\it
sharp interface limit} --- of $(P^\ep)$ as $\ep \to 0$.

The reaction-diffusion equation in problem $(P^\ep)$ arises from
the hyperbolic space-time rescaling $u^\ep(t,x):=u\left(\frac t
\ep, \frac x \ep\right)$ of the heterogeneous equation
\begin{equation}\label{monostable}
\partial _t u=\Delta
u+ f( x,u).
\end{equation}
Let us emphasize that the understanding of the long time behavior
of \eqref{monostable} is not equivalent to that of the sharp
interface limit of $(P^\ep)$. Roughly speaking, the former one
deals with the stabilization of the interface into a predetermined
shape after a long time, whereas the latter one keeps the memory
of the shape of the initial data. In other words, the singular
limit procedure describes some transient states, during which
geometry is quite relevant.

\medskip
Let us now state the assumptions on the nonlinearity $f(x,u)$. Let
$L_1$,...,$L_N$ be given positive constants. A function $h:\R ^ N
\to \R$ is said to be {\it periodic} if
$$
h(x_1,...,x_k+L_k,...,x_N)=h(x_1,...,x_N),
$$
for all $1\leq k \leq N$, all $(x_1,...,x_N)\in \R^N$. In such
case, $(0,L_1)\times\cdots \times (0,L_N)$ is called the cell of
periodicity. Through this work, we  assume that
\begin{equation}\label{periodicity}
\text{ for all } u\in \R_+,\, f(\cdot,u): \R^N \to \R \text{ is
periodic}.
\end{equation}

Our second main assumption on the nonlinearity $f $ is the
following.

\begin{assumption}[Monostable nonlinearity]\label{hyp:monostable} The function $f:\R ^N \times \R_+\to \R$ is of class $C^{1,\alpha}$ in $(x,u)$
and~$ C^2$ in $u$, and nonnegative on $\R^N\times [0,1]$.
Concerning the steady states of the periodic
equation~\eqref{monostable}, we assume that
\begin{enumerate}[$(i)$]
\item the constants 0 and 1 are steady states (that is,
$f(\cdot,0)\equiv f (\cdot,1) \equiv 0$ in $\R^N$); \item $\forall
u \in (0,1), \ \exists x \in \R^N, \ \  f(x,u) >0$. \item there
exists some $\rho
>0$ such that $f (x,u)$ is nonincreasing with respect to $u$ in
the set $\R^N \times (1-\rho,1]$.
\end{enumerate}
\end{assumption}

Notice that, if $0\leq p(x)\leq 1$ is a periodic stationary state,
then $p\equiv 0$ or $p\equiv 1$. Indeed, since $f(x,p)\geq 0$, the
strong maximum principle enforces $p$ to be identically equal to
its minimum, thus constant and, by $(ii)$, the constant has to be
0 or 1. Hence, under the above hypotheses,
equation~\eqref{monostable} is often referred to as the monostable
equation. Typical examples are of the form $f(x,u) = p(x)
\tilde{f} (u)$, where $p(x)$ is positive and periodic, and
$\tilde{f}$ is a homogeneous nonlinearity possibly of the
following types: $\tilde f_1(u)=u(1-u)$ (Fisher-KPP),
$\tilde{f}_2(u)= u^r (1-u)$ with $r>1$ (weak Allee effect),
$\tilde{f}_3(u) = e^{-1/u} (1-u)$ (Arrhenius nonlinearity), or
$\tilde f _4(u)=u(e^{1-u}-1)$ (Nicholson's blowflies equation).

The monostable problem \eqref{monostable} arises in various fields
of physics and the life sciences, and especially in population
dynamics models where propagation phenomena are involved. Indeed,
a particular feature of this equation is the formation of
traveling fronts, that is particular solutions describing the
transition at a constant speed from one stationary solution to
another one. Such solutions have proved in numerous situations
their utility in describing the dynamics of a population modelled
by a reaction-diffusion equation.

Equation \eqref{monostable} is a heterogeneous version of the
reaction-diffusion equation
\begin{equation}\label{homogene}
\partial _t u=\Delta u +f(u),
\end{equation}
with $f$ of the monostable type. Among monostable nonlinearities,
one can distinguish the ones satisfying the Fisher-KPP assumption,
namely $u\mapsto \frac{f(u)}u$ is maximal at 0, the most famous
example $f(u)=\tilde f _1(u)=u(1-u)$ being introduced by Fisher
\cite{Fish} and Kolmogorov, Petrovsky and Piskunov
\cite{Kol-Pet-Pis} to model the spreading of advantageous genetic
features in a population. The KPP assumption means that the growth
is only slowed down by the intra-specific competition, so that the
growth per capita is maximal at small densities. Due for instance
to the lack of genetic diversity at low density, this assumption
may be unrealistic. To take into account such a  {\it weak Allee
effect}, one may use the growth function $f(u)=\tilde f
_2(u)=u^r(1-u)$, $r>1$. The nonlinearity $f(u)=\tilde f
_4(u)=u(e^{1-u}-1)$ is commonly used \cite{Gur-Bly-Nis} to explain
oscillations of a population of Australian sheep blowflies, {\it
Lucilia Cuprina}, described by Nicholson \cite{Nic}. Let us notice
that our work stands in the class of monostable nonlinearities,
and therefore covers all these examples coming from population
dynamics models, and the Arrhenius case $f(u)=\tilde{f}_3(u) =
e^{-1/u} (1-u)$ which comes from combustion models.

Nevertheless, the environment is rarely homogeneous and may depend
in a non trivial way on the position in space (patches, periodic
media, or more general heterogeneity...), so that one should take
into account heterogeneities. We refer to the seminal book of
Shigesada and Kawasaki \cite{Shi-Kaw}, and the enlightening
introduction in \cite{Ber-Ham-Roq1} where the reader can find very
precise and various references. For example such heterogeneities
are very pertinent in some epidemiology models, where different
treatments (antibiotics or insecticides) are tested, aiming at
finding an optimal combination.

In a periodic framework, traveling fronts in the homogeneous
equation \eqref{homogene} are replaced by the so-called {\it
pulsating} traveling fronts in the periodic equation
\eqref{monostable} (see below for details). As far as the rescaled
equation in $(P^\ep)$ is concerned, fronts become sharper as $\ep
\to 0$, and we therefore have to deal with the so-called
interfaces. Also, as explained above, the singular limit analysis
of \eqref{monostable} describes a transient state where the
geometry of the initial habitat of the population is an insightful
information.

\medskip  In this paper, we aim at looking at the way those
interfaces are generated and propagate, hence providing some
accurate connection between the behavior of solutions $u^\ep(t,x)$
in the fast reaction and low diffusion regime and some free
boundary problem. One of the originality of this work is that we
allow the equation to be spatially heterogeneous, which as recalled above is essential in realistic biological models. More precisely, we restrict
ourselves to the spatially periodic case, which provides
insightful information on the role and influence of the
heterogeneity on the propagation, as well as a slightly more
common mathematical framework.

We will describe in subsection \ref{ss:pulsating-fronts} what is
known as far as front-like solutions of~\eqref{monostable} are
concerned. In particular, we will see that the outcome of the
heterogeneity is some new dynamics, that do not appear in the
homogeneous case, where the speed of the propagation depends on
its direction. This feature is the origin of new technical
difficulties when retrieving the interface motion.

\medskip

As far as initial data $g(x)$ appearing in $(P^\ep)$ are
concerned, we make the following hypotheses.
\begin{assumption}[Structure of initial data]\label{H1}
\begin{itemize}
\item [$(i)$] Let $\Omega_0$ be a nonempty, open and bounded set of
$\R ^N$. Let $\widetilde g:\overline{\Omega_0}\to [0, 1 )$ be a
map of the class $C^2$ on $\overline{\Omega_0}$, positive on
$\Omega _0$ and such that $\widetilde g(x)=0$ for all
$x\in\partial \Omega _0$. Define the map $g:\R^N\to \R$ by
\begin{equation*}
g(x)=\begin{cases} \widetilde g(x)&\text{ if } x\in\overline{\Omega_0}\\
0&\text{ if } x\notin \overline{\Omega_0}\,.\end{cases}
\end{equation*}
\item[$(ii)$] We assume that $\Omega _0$ is convex and has a smooth
boundary $\Gamma _0:=\partial \Omega _0$.
\end{itemize}
\end{assumption}

Notice that the assumption $g(x) < 1$ becomes unnecessary if one
assumes further that there is no steady state for
\eqref{monostable} above $1$. Also, rather than compactly
supported initial data, one may allow $g(x)$ to have tails that
are \lq\lq consistent" with those of the pulsating fronts (see
\cite{A-Duc2} for the homogeneous case with \lq\lq tails"). For
the sake of simplicity, we do not consider here such cases. The
convexity assumption $(ii)$ will allow to describe explicitly the
limit interface (obtained via a Hamilton-Jacobi approach) in
Proposition \ref{prop:motion?} and then to use a family of planar
supersolutions in Section \ref{s:control-above}.

\medskip

Before stating our results, let us now comment on related works.
First, there is a large literature on the singular limit of
(generalizations of)
\begin{equation}\label{souga}
 \partial _t u^\ep= \ep \Delta u^\ep+\displaystyle \frac 1 \ep
 f\left(x,u^\ep\right).
\end{equation}
Observe that \eqref{souga} arises after a hyperbolic rescaling of
$$
\partial _t u=\Delta u+ f(\ep x,u),
$$
whereas Problem $(P^\ep)$ under consideration follows from
\eqref{monostable}. First results are due to Freidlin \cite{Frie,
Frie2} using probabilistic methods. Later, Evans and Souganidis
\cite{Eva-Sou} used PDE technics, Hamilton-Jacobi framework to be
more precise, to study
 \eqref{souga}. In this context, we also
refer to \cite{Bar-Eva-Sou}, \cite{Bar-Sou} and, for an overview,
to \cite{Sou}.   Let us also mention the related work
\cite{Maj-Sou} which is linked with homogenization processes
\cite{Lio-Pap-Var}. As far as (generalizations of) the considered
problem $(P^\ep)$ is concerned, we refer to \cite[Section
9]{Lio-Sou} where Hamilton-Jacobi and homogenization technics are
combined. Nevertheless, notice that all these results hold under
the KPP assumption, that is $f(x,u)\leq f_u(x,0)u$, whereas we
stand in the larger class of monostable nonlinearities.

In the homogeneous case $f(x,u)=f(u)$, the sharp interface limit
of \eqref{souga} has been recently revisited using specific
reaction-diffusion tools, such as the comparison principle and
traveling wave solutions, which allows to capture accurate convergence rates~\cite{A-Duc, A-Duc2}. Hence, the
introduction of a delay effect has been handled in \cite{A-Duc3},
via such methods.

Our analysis of the introduction of heterogeneity in $(P^\ep)$
stands mainly in this latter framework. It relies on accurate
\lq\lq local" subsolutions combined with improved spreading speeds
properties \cite{A-Gil}, and on a family of planar supersolutions
whose envelop solves the limit Hamilton-Jacobi equation.

\section{Some known results}\label{s:known_results}

Before stating our main results in Section \ref{ss:main-results},
we need to say a few words on monostable pulsating fronts and
spreading speeds (in subsection \ref{ss:pulsating-fronts}), and on
the limit free boundary problem $(P^0_{HJ})$ (in subsection
\ref{ss:free-boundary-pbs}), which is expected to describe the
motion of the transition layers of the solutions $u^\ep(t,x)$ of
$(P^\ep)$, as $\ep \to 0$.

\subsection{Monostable pulsating fronts and spreading properties}\label{ss:pulsating-fronts}

The definition of the so-called pulsating traveling wave was
introduced by Xin \cite{Xin} in the framework of flame
propagation. It is the natural extension, in the periodic
framework, of classical traveling waves. Due to the interest of
taking into account the role of the heterogeneity of the medium on
the propagation of solutions, a lot of attention was later drawn
on this subject. As far as monostable pulsating fronts are
concerned, we refer to the seminal works of Weinberger
\cite{Wein02}, Berestycki and Hamel \cite{Ber-Ham-02}. Let us also
mention \cite{Ber-Ham-Roq2}, \cite{Ham}, \cite{Ham-Roq},
\cite{Nad-09} for related results.

\medskip

 For the sake of completeness,
let us first recall the definition of a pulsating traveling wave
for the monostable equation~\eqref{monostable}, as stated in
\cite{Ber-Ham-02}.

\begin{defi}[Pulsating traveling wave]\label{def:puls}
A pulsating traveling wave solution, with speed $c >0$ in the
direction $n \in \mathbb{S}^{N-1}$, is an entire solution $u(t,x)$
--- $t\in\R$, $x\in \R ^N$--- of \eqref{monostable} satisfying
$$
\forall k \in \prod_{i=1}^N L_i \mathbb{Z} , \qquad u(t
,x)=u\left(t + \frac{k\cdot n}{c} ,x+k\right),
$$
for any $t \in \R$ and $x\in \R^N$, along with the asymptotics
$$
u(-\infty,\cdot)=0  < u (\cdot , \cdot) < u(+\infty,\cdot)
=1,
$$
where the convergences in $\pm \infty$ are understood to hold
locally uniformly in the space variable.
\end{defi}

One can easily check that, for any $c>0$ and $n \in
\mathbb{S}^{N-1}$, $u(t,x)$ is a pulsating traveling wave with
speed $c$ in the direction $n$ if and only if it can be written in
the form
$$
u(t,x) = U (x\cdot n -ct,x),
$$
where~$U(z,x)$ --- $z\in \R$, $x\in \R^N$--- satisfies
$$
\text{ for all } z\in \R,\, U(z,\cdot): \R^N \to \R \text{ is
periodic},
$$
\begin{equation*}
 U (-\infty, \cdot) = 1  < U (\cdot , \cdot) <   U (+\infty,\cdot)
 =0 \quad \text{ uniformly w.r.t. the space variable},
\end{equation*}
along with the following equation
\begin{equation}\label{eq-tw}
(\partial_{zz}+ \Delta_x ) U+2\nabla _x \partial _z U \cdot
n+c\partial _z U+f(x,U)=0\quad \text{ on } \R\times \R ^N.
\end{equation}

We can now recall the result of \cite{Ber-Ham-02}, \cite{Wein02}, on existence of pulsating traveling waves for
the spatially periodic monostable equations: in any direction
there is a minimal speed $c^* (n)>0$ which allows existence.
Precisely, the following holds.

\begin{theo}[Monostable pulsating fronts, \cite{Ber-Ham-02}, \cite{Wein02}]
Assume that $f$ is of the spatially periodic monostable type, i.e.
$f$ satisfies \eqref{periodicity} and
Assumption~\ref{hyp:monostable}.

Then for any $n \in \mathbb{S}^{N-1}$, there exists $c^* (n) >0$
such that traveling waves with speed $c$ in the $n$-direction
exist if and only if $c \geq c^* (n)$. Furthermore, any pulsating
traveling wave is increasing in time.
\end{theo}

In the Fisher-KPP case the continuity of the velocity map
$n\mapsto c^*(n)$, even if not explicitly stated, seems to follow
from the characterization of $c^*(n)$ (see \cite{Wein02},
\cite{Ber-Ham-02}). In the more general monostable case, such a
property was recently proved.

\begin{theo}[Continuity of minimal speeds, \cite{A-Gil}]\label{th:continuity}
The mapping $n\in \mathbb S ^{N-1}\mapsto c^*(n)$ is continuous.
\end{theo}

The introduction of these pulsating traveling waves was motivated
by their expected role in describing the large time behavior of
solutions of \eqref{monostable} for a large class of initial data.
In this context, let us recall the seminal result of
\cite{Wein02}: for any planar-like initial data in some direction
$n$, the associated solution of \eqref{monostable} spreads in the
$n$ direction with speed $c^* (n)$. Actually, for our singular
limit analysis, it turns out that we need the stronger property
that this spreading is uniform with respect to the direction $n$. This was
the purpose of our previous work \cite{A-Gil}.

\begin{theo}[Uniform spreading, \cite{A-Gil}]\label{th:unif_spreading}
Assume that $f$ is of the spatially periodic monostable type, i.e.
$f$ satisfies \eqref{periodicity} and
Assumption~\ref{hyp:monostable}. Let a family of nonnegative
initial data $(u_{0,n})_{n \in \mathbb{S}^{N-1}}$ such that
\begin{equation*}
\exists C >0, \qquad \forall n \in \mathbb{S}^{N-1}, \qquad x
\cdot n \geq C \Longrightarrow u_{0,n} (x) = 0,
\end{equation*}
\begin{equation*}
\exists \mu >0 , \qquad \exists K >0, \qquad \inf_{n \in
\mathbb{S}^{N-1}, \; x \cdot n \leq -K } u_{0,n} (x) \geq \mu,
\end{equation*}
\begin{equation*}
\inf_{n \in \mathbb{S}^{N-1}} \inf_{x \in \R^N} 1 - u_{0,n} (x)
>0.
\end{equation*}
 We denote by $(u_n)_{n \in \mathbb{S}^{N-1}}$ the associated
family of solutions of \eqref{monostable}.

Let $\alpha >0$ and $\eta >0$ be given. Then, there exists $\tau
>0$ such that for all $t \geq \tau$,
\begin{equation}\label{conclusion1}
\sup_{n \in \mathbb{S}^{N-1}} \ \sup_{x \cdot n \leq (c^*(n) -
\alpha)t} |1 -u_n (t,x)| \leq  \eta,
\end{equation}
\begin{equation}\label{conclusion2}
\sup_{n \in \mathbb{S}^{N-1}} \ \sup_{x \cdot n \geq (c^*(n) +
\alpha)t} u_n (t,x) \leq  \eta.
\end{equation}
\end{theo}

Let us notice that, under suitable assumptions such as those in
\cite{Ber-Ham-02}, \cite{A-Gil}, the above results are also
available for more general spatially periodic and monostable equations which
include heterogeneous diffusion and advection terms. We restrict ourselves to Problem
$(P^\ep)$ to simplify the presentation, but our argument easily extends to such a framework.

\subsection{On limit free boundary problems}\label{ss:free-boundary-pbs}

We recall that we aim at investigating the $\ep \to 0$ limit of
$u^\ep(t,x)$ the solution of $(P^\ep)$. Then the limit solution
$\tilde u ^\ep (t,x)$ will be a step function, taking the value
$1$ on one side of a moving interface which we will denote by $\Gamma_t$, and $0$ on the
other side. This sharp interface, if smooth, obeys the law of motion
\[
 (P^0) \quad\begin{cases}
  V_n=c^*(n) \quad \text { on } \Gamma_t& \vspace{5pt}\\
 \Gamma_t\big|_{t=0}=\Gamma_0,
 \end{cases}
\]
where $V_n$ denotes the normal velocity of $\Gamma  _t$ in the
exterior direction $n$, the unit outer normal of $\Gamma _t$ at
each point $x\in \Gamma _t$. Here $c^*(n)$ denotes the minimal
speed of the underlying monostable pulsating wave traveling in the
$n$-direction.

As we only know the mapping $n\mapsto
c^*(n)$ to be continuous, the smoothness of the interface, and hence the well posedness of $(P^0)$, is not guaranteed even for small positive times.

\medskip

A classical way to overcome the lack of smoothness is to define the limit interface via the level sets of the
viscosity solution of a Hamilton-Jacobi problem
\[
(P^0_{HJ})
\quad \begin{cases}
 \partial _t w+|\nabla w|c^*\left(\frac{\nabla w}{|\nabla w|}\right)=0 &\text{in }(0,\infty)\times \R ^N  \vspace{5pt}\\
 w(0,x)=w_0(x) &\text{in }\R ^N.
 \end{cases}
\]
Here $w_0:\R^N\to
\R$ is any uniformly continuous function  such that
\begin{equation}\label{au-debut}
\Omega _0=\{x:\; w_0(x)<0\},\quad \Gamma _0=\{x:\; w_0(x)=0\}.
\end{equation}
Thanks to the continuity of $c^* (n)$ with respect to $n \in
\mathbb{S}^{N-1}$, namely Theorem \ref{th:continuity}, the
Hamilton-Jacobi problem admits a unique viscosity solution $w\in
C((0,\infty)\times \R ^N)$, and
\begin{equation*}%\label{apres}
\Omega _t:=\{x:\; w(t,x)<0\},\quad \Gamma _t:=\{x:\; w(t,x)=0\}
\end{equation*}
do not depend on the choice of $w_0$ as above (see Theorems~4.3.5
and 4.3.6 in \cite{Gig}). As long as the solution of $(P^0)$ has a
smooth solution, both motions coincide, which is why we still
denote it by $\Gamma_t$. However, the Hamilton-Jacobi approach
does not require smoothness as $(P^0)$ does, and therefore enables
to define the zero level set $\Gamma _t$ as the limit interface
for all $t\geq 0$.

The literature on this level set approach via viscosity solutions
of Hamilton-Jacobi equations is rather large. The reader may
consult \cite{Che-Gig-Got} or the book of Giga \cite{Gig} and the
references therein.

\medskip

Thanks to the convexity of the initial set $\Omega_0$, a so-called
Hopf formula~\cite{Hopf} is actually available and provides an
explicit depiction of the motion, as stated in the following
result.

\begin{prop}[The limit interface explicitly]\label{prop:motion?}
Let Assumption \ref{H1} $(ii)$ hold. Let the limit interface
$\Gamma _t$ be defined via the Hamilton-Jacobi problem
$(P^0_{HJ})$ as above.

Then, for all time $t \geq 0$, the set $\Gamma_t$ is the zero
level set of the convex function
$$v (t,x) := \max_{y \in \Gamma_0}\ (x-y).n_y - c^* (n_y) t,$$
where $n_y$ denotes the outward unit normal vector of $\Gamma_0$
at point $y$. In particular, for all time $t\geq 0$, the set
$\Gamma_t$ remains sharp, in the sense that it does not develop an
interior, and the bounded domain $\Omega_t$ delimited by
$\Gamma_t$ remains convex.
\end{prop}

Roughly speaking, this proposition means that the motion can be
described by first looking at $\Gamma_0$ as the envelop of some
half-spaces, and by then letting each of those half-spaces move at
the speed $c^* (n)$ corresponding to its normal direction. We
refer to \cite{Bardi-Evans} where the Hopf formula was revisited
in the context of viscosity solutions, and obtained using the more
general theory of differential games. We propose a direct proof of
Proposition~\ref{prop:motion?} in
subsection~\ref{ss:interface_motion}.

\section{Main result}\label{ss:main-results}

We are now in the position to state our main result of convergence
of $(P^\ep)$ to the interface motion defined via the level sets of
solutions of
 $(P^0_{HJ})$. Together with Proposition~\ref{prop:motion?}, the theorem below provides a precise depiction of the shape of solutions or, in other words, of the expansion of the habitat of the population.

\begin{theo}[Convergence to a propagating interface]\label{THEO-conv}
Let the nonlinearity $f$ be of the spatially periodic monostable
type, i.e. $f$ satisfies \eqref{periodicity} and
Assumption~\ref{hyp:monostable}, and let the initial data $g$ in
Problem $(P^\ep)$ satisfy Assumption~\ref{H1}. For any $\ep>0$,
let $u^\ep:[0,\infty)\times \R ^N \to \R$ be the solution of
$(P^\ep)$.  Let $\Gamma _t$ and $\Omega _t$ be defined via the
Hamilton-Jacobi problem $(P^0_{HJ})$ as in subsection
\ref{ss:free-boundary-pbs}.

Then, the following convergence results hold.
\begin{itemize}
\item [$(i)$] For any $0<\tau\leq T < +\infty$ and small $\beta>0$,
we have
\begin{equation*}
 \sup _{\tau \leq t \leq T} \; \sup _{ \{ x : d(t,x) \leq - \beta \}}\; \left|1 -u^\ep(t,x)\right|\to 0 \quad \text{ as } \ep \to 0;
\end{equation*}
\item [$(ii)$] For any $0<  T < +\infty$ and small $\beta>0$,
we have
\begin{equation*}
 \sup _{0 \leq t \leq T} \; \sup _{\{x:
d(t,x)\geq \beta\}}\; |u^\ep(t,x)|\to 0 \quad \text{ as } \ep \to
0.
\end{equation*}
\end{itemize}
Here $d (t,\cdot)$ denotes the signed distance to the set
$\Gamma_t$, which is chosen to be negative in $\Omega_t$ and
positive  in $\R ^N \setminus (\Gamma _t \cup {\Omega_t})$.
\end{theo}

The rest of the paper is devoted to the proof of
Theorem~\ref{THEO-conv} and is organized as follows.

We start, in Section \ref{s:themotion}, by some results on the
motion of the limit interface which are crucial to our analysis of
the parabolic problem $(P^\ep)$, but are also of independent
interest for the Hamilton-Jacobi problem $(P^0_{HJ})$. On the one
hand, we prove Proposition~\ref{prop:motion?}, hence providing an
explicit description of the limit interface. On the other hand, we
approximate the motion defined via $(P_{HJ}^0)$ by a smooth
motion, which preserves all its essential geometric properties.

 To prove the control from below $(i)$ of Theorem \ref{THEO-conv}, we
distinguish two regimes. First, we prove in
Section~\ref{s:generation-below} the emergence of transition
layers for $u^\ep(t,x)$ in very small times. The propagation of
the layers (from below) that occurs in later times is then studied
in Section~\ref{s:propagation-below}. The heterogeneity rises some
technical difficulties since pulsating fronts depend non trivially
on the direction of propagation. Roughly speaking, we construct
\lq\lq local" subsolutions and combine the uniform spreading
properties of Theorem \ref{th:unif_spreading} with an iteration
procedure. The construction of such subsolutions requires
smoothness of the interface, which insures that the motion is
locally governed by the planar dynamics of the
 rescaled equation~\eqref{monostable}. Hence, we actually apply
 the above procedure to the smooth approximated motion defined in
 Section \ref{s:themotion}.

Last, in Section \ref{s:control-above},
to prove the control from above $(ii)$ of Theorem \ref{THEO-conv}, we construct a family of
planar supersolutions
--- whose envelop \eqref{def:v} coincides with the explicit characterization of
Proposition~\ref{prop:motion?} --- and use again the uniform
spreading properties of Theorem \ref{th:unif_spreading}.

\section{Some results on the motion of the limit interface}\label{s:themotion}

In this section, we are only concerned with the limit interface
motion $(P^0_{HJ})$. We first prove the explicit description of
Proposition~\ref{prop:motion?}, and then proceed to an
approximation of the motion $(P^0_{HJ})$ by a smooth motion. As
mentioned before, smoothness will play an
 essential role in the convergence of solutions of $(P^\ep )$, and more specifically in Section~\ref{s:propagation-below}.

\subsection{Characterization of the motion}\label{ss:interface_motion}

We begin by recalling that
\begin{equation*}\label{def:v}
v(t,x):=\max _{y\in \Gamma _0} \; \{(x-y).n_y -c^*(n_y)t\},
\end{equation*}
where  $n_y$ is the outward unit normal of $\Gamma_0$ at point
$y$. The zero level sets of $v(t,x)$ are obtained by \lq\lq
intersecting all the half-planes arising from $y \in \Gamma _0$
and propagating with speed $c^*(n_y)$ in direction $n_y$". We will
prove that, at least for its level sets lying above some small
$-\delta <0$, the function $v$ is a viscosity solution of the
Hamilton-Jacobi problem $(P^0_{HJ})$. As the motion of interface
is defined by the zero level set of the viscosity solution, this
will be enough  to infer that its zero level set defines the
appropriate interface $\Gamma_t$, that is Proposition
\ref{prop:motion?}.

\begin{rem}
Write $v(t,x)=\max _{y\in \Gamma _0} \psi(t,x,y)$ where
$$
\psi: (t,x,y)\in (0,\infty)\times \R ^N \times \Gamma _0 \mapsto
(x-y).n_y -c^*(n_y)t,
$$
is continuous with respect to $y\in \Gamma _0$, smooth and convex
(since linear) with respect to $t>0$ and $x\in \R ^N$. For a given
$(t,x)\in(0,\infty)\times \R ^N$, let us denote by $Y(t,x)$ the
set of $y\in\Gamma _0$ that maximize $\psi(t,x,\cdot)$, that is
$$
Y(t,x)=\{y\in\Gamma _0:\; v(t,x)=\psi(t,x,y)\}.
$$
If, for a given $(t_0,x_0)$, the set $Y(t_0,x_0)$ reduces to a
singleton $y_0$ then it follows from classical results of convex
analysis (see \cite[Corollary 4.4.5]{Hir-Lem}) that $v$ is
differentiable at $(t_0,x_0)$, and
$$
\partial _t v(t_0,x_0)=\partial _t
\psi(t_0,x_0,y_0)=-c^*(n_{y_0}),\quad \nabla _x v(t_0,x_0)=\nabla
_x\psi(t_0,x_0,y_0)=n_{y_0},
$$
so that $v$ satisfies the Hamilton-Jacobi equation $\partial _t v+
|\nabla v|c^*\left(\frac{\nabla v}{|\nabla v|}\right)= 0$ in the
classical sense at $(t_0,x_0)$. However, we have to deal with the
case where $Y(t_0,x_0)$ is not a singleton. As we will see, this
can be performed in the set $\{v\geq -\delta\}$ for some small
enough $\delta>0$, and requires to cut-off the set $\{v<
-\delta\}$.
\end{rem}

We prove the following proposition, of which
Proposition~\ref{prop:motion?} is an immediate corollary.

\begin{prop}[A solution of the Hamilton-Jacobi problem]\label{prop:sol-HJ} For any small enough $\delta >0$, the function
\begin{equation*}%\label{def:wdelta}
 v_\delta (t,x):=\max( -\delta;v(t,x)),
\end{equation*}
is a (viscosity) solution of the equation of the limit problem
$(P^0_{HJ})$, that is
\begin{equation}\label{eq-HJ}
\partial _t v_\delta + |\nabla v_\delta |c^*\left(\frac{\nabla v_\delta}{|\nabla
v_\delta |}\right)= 0, \quad \text{ in } (0,\infty)\times\R^N,
\end{equation}
and, by convexity, $v_\delta(0,x)$ is an admissible initial data
for $(P^0_{HJ})$ in the sense of \eqref{au-debut}.
\end{prop}

\begin{proof}
Recall that at time $t=0$,  $\Gamma_0$ is a smooth hypersurface,
and that the bounded set $\Omega _0$ delimited by $\Gamma _0$ is
convex. Hence, for $\delta
>0$ small enough, one can define a smooth hypersurface $\Gamma_0^{-\delta}$ as
$$\Gamma_0^{-\delta} := \{ x \in \R^N \ : \ d(0,x)=-\delta\} = \{ y - \delta n_y  \ : y \in \Gamma_0 \},$$
where $d (0,\cdot)$ denotes the signed distance to $\Gamma_0$,
which is negative in the bounded set $\Omega_0$, and positive in
$\R ^N \setminus \overline {\Omega _0}$. Notice also that, when $x\in\Omega _0$, we can write $v(0,x)=-\min _{y\in\Gamma _0} dist(x,H_y)$, where
 $H_y$ is the hyperplane going through $y$ and with normal vector $n_y$. As a result, the convexity assumption yields
$$
\Gamma_0^{-\delta} =\{ x \in \R^N \ : \ v(0,x) = -\delta \}.
$$
In particular, since the function $v$ is convex with respect to $x$, the bounded set
$\Omega_0^{-\delta}$ delimited by $\Gamma_0^{-\delta}$ is still
convex. Moreover, as we have chosen $\delta$ small enough so that $\Gamma_0^{-\delta}$ is smooth, it is straightforward that the outward unit normal vector of $\Gamma_0^{-\delta}$
at $y - \delta n_y$ is also $n_y$. Therefore, by some slight abuse
of notation, when $y \in \Gamma_0^{-\delta}$, $n_y$ will denote
the outward unit normal vector of $\Gamma_0^{-\delta}$ at point
$y$. Then
\begin{eqnarray*}
v_\delta (t,x) & = & \max \{ - \delta , \max_{y \in \Gamma_0} \{ (x-y).n_y - c^* (n_y) t \} \} \\
& = & \max \{ -\delta , \max_{y \in \Gamma_0 } \{ (x-(y - \delta n_y)). n_y - c^* (n_y) t - \delta \} \} \\
& = & \max \{ 0 , \max_{y\in \Gamma_0^{-\delta}} \{ (x - y) . n_y
- c^* (n_y) t\} \} - \delta.
\end{eqnarray*}
Therefore, $v_\delta (t,x)$ is a solution of $(P_{HJ}^0)$ if and
only if
$$\bar v_\delta (t,x) := \max \{0 , \max_{y \in \Gamma_0^{-\delta}} \{ (x-y).n_y - c^* (n_y) t \} \}$$
is. For convenience, denote
$$
\psi^\delta : (t,x,y)\in (0,\infty)\times \R ^N \times
\Gamma^{-\delta}_0 \mapsto (x-y).n_y -c^*(n_y)t,
$$
which is continuous with respect to $y\in \Gamma^{-\delta}_0$,
smooth and linear with respect to $t>0$ and $x\in \R ^N$, and
introduce also
$$w_\delta (t,x) := \max_{y \in \Gamma_0^{-\delta}} \ \psi^\delta (t,x,y),$$
so that $\bar v_\delta (t,x) = \max (0,w_\delta (t,x) )$.\\

Let us now prove that $\bar v_\delta$ is a solution of
\eqref{eq-HJ}. First, the null function and each function
$(t,x)\mapsto \psi^\delta (t,x,y)$ solve \eqref{eq-HJ} so that
$\bar v_\delta (t,x)$ --- as a supremum of solutions --- is a
viscosity subsolution of \eqref{eq-HJ}.

To prove that $\bar v_\delta (t,x)$ is also a supersolution, let
$\varphi$ be a smooth test function such that $\bar v_\delta
-\varphi$ has a zero local minimum at some point
$(t_0,x_0)\in(0,\infty)\times \R ^N$. We need to prove that
\begin{equation}\label{varphi}
\partial _t \varphi (t_0,x_0)+ |\nabla \varphi (t_0,x_0)|c^*\left(\frac{\nabla \varphi (t_0,x_0)}{|\nabla
\varphi (t_0,x_0)|}\right)\geq 0.
\end{equation}
If $w_\delta (t_0,x_0)< 0$, then $\bar v_\delta \equiv 0$ in a
neighborhood of $(t_0,x_0)$ and \eqref{varphi} is clear. Let us
now assume $0\leq  w_\delta (t_0,x_0)  =\bar v_\delta (t_0,x_0)$.
Since $\bar v_\delta -\varphi$ has a zero local minimum at
$(t_0,x_0)$, the time-space gradient of $\varphi$ at $(t_0,x_0)$
must belong to the time-space subdifferential of $\bar v_\delta$
at $(t_0,x_0)$, which is given by
$$
\partial \bar v_\delta (t_0,x_0)=
\begin{cases}
\partial w_\delta (t_0,x_0) & \text{ if  } w_\delta (t_0,x_0)>0 \\
\text{Co } \{(0_\R,0_{\R^N}) \cup \partial w_\delta (t_0,x_0)  \}
& \text{ if } w_\delta (t_0,x_0)=0,
\end{cases}
$$
where $\text{Co } A$ denotes the convex hull of the set $A$. It
also follows from \cite[Theorem 4.4.2]{Hir-Lem} that
$$
\partial w_\delta (t_0,x_0)=\text{Co } \{(\partial _t \psi^\delta (t_0,x_0,y),\nabla
_x\psi^\delta (t_0,x_0,y))=(-c^*(n_y),n_y)\in \R \times \R
^{N}:\,y\in Y(t_0,x_0)\},
$$
where $Y(t_0,x_0)$ is the set of $y\in \Gamma^{-\delta}_0$ that
maximize $\psi^\delta (t_0,x_0,\cdot)$. Hence, in any case, one
can write
$$
\partial _t \varphi (t_0,x_0)=\sum_{i=1}^p -\lambda _i
c^*(n_i),\quad \nabla \varphi (t_0,x_0)=\sum_{i=1}^p \lambda _i
n_i,
$$
for some $y_1$,...,$y_p$ in $Y(t_0,x_0)$, and $n_i$ the outward
unit normal of $\Gamma^{-\delta}_0$ at point $y_i$, and some
nonnegative $\lambda _1$,...,$\lambda _p$ such that $\sum _{i=1}^p
\lambda _i\leq 1$. Therefore our goal \eqref{varphi} is recast as
\begin{equation}\label{concavite-partielle}
c^*\left(\frac{\sum_{i=1}^p \lambda _i n_i}{|\sum_{i=1}^p \lambda
_i n_i|}\right)\geq \frac{\sum_{i=1}^p \lambda _i
c^*(n_i)}{|\sum_{i=1}^p \lambda _i n_i|}.
\end{equation}
Let us define
$$
n_0 :=\frac{\sum_{i=1}^p \lambda _i n_i}{|\sum_{i=1}^p \lambda _i
n_i|}\in \mathbb S ^{N-1},
$$
and pick a $y_0 \in \Gamma^{-\delta}_0$ such that $n_{y_0}= n_0$.
Note that such a $y_0$ necessarily exists from the smoothness of
the bounded hypersurface $\Gamma^{-\delta}_0$. One must then have
$$\psi^\delta (t_0,x_0,y_0)=(x_0-y_0).n_0-c^*(n_0)t_0\leq w_\delta (t_0,x_0),$$
so that
\begin{eqnarray*}
c^* (n_0) t_0 & \geq & (x_0 - y_0).n_0-w_\delta(t_0,x_0)\\
& = & \frac {\sum _{i=1}^p \lambda _i
(x_0-y_0).n_i}{|\sum_{i=1}^p \lambda _i n_i|}-w_\delta (t_0,x_0)\\
& \geq & \frac {\sum _{i=1}^p \lambda _i
(x_0-y_i).n_i}{|\sum_{i=1}^p \lambda _i n_i|}-w_\delta (t_0,x_0).
\end{eqnarray*}
Here we used the convexity of $\Omega _0 ^{-\delta}$, so that
$(y_i-y_0).n_i\geq 0$ for all $1\leq i \leq p$. Next, as each
$y_i$ belongs to $Y(t_0,x_0)$, we have $w_\delta
(t_0,x_0)=(x_0-y_i).n_i-c^*(n_i)t_0$, so that
$$
c^*(n_0)t_0 \geq \frac{\sum_{i=1}^p \lambda _i
c^*(n_i)}{|\sum_{i=1}^p \lambda _i n_i|}t_0+w_ \delta (t_0,x_0)
\left(\frac {\sum _{i=1}^p \lambda _i} {|\sum_{i=1}^p \lambda _i
n_i|}-1\right)\geq \frac{\sum_{i=1}^p \lambda _i
c^*(n_i)}{|\sum_{i=1}^p \lambda _i n_i|}t_0,
$$
since $w_\delta (t_0,x_0)\geq 0$ (notice that this is where it
fails if no cut-off is performed). This
proves~\eqref{concavite-partielle} and concludes the proof of
Proposition~\ref{prop:sol-HJ}.
\end{proof}

\subsection{Regularization of the motion}\label{ss:regular}

We now construct, by the vanishing viscosity method, a smooth
hypersurface $\Gamma^\alpha_t$ which approximates the interface
$\Gamma_t$ as $\alpha \to 0$. Moreover, the motion of this smooth
hypersurface is always \lq\lq slower'' than that of the original
interface $\Gamma_t$: this will allow us, in Section
\ref{s:propagation-below}, to construct subsolutions of $(P^\ep)$
which fully cover the bounded set delimited by $\Gamma^\alpha_t$.

\begin{prop}[Approximated smooth motion]\label{regularis}
Fix $\alpha_0>0$ small enough and, for any $0 < \alpha \leq \alpha_0$, let $F^\alpha: \R^N\to \R$ be a smooth
function such that
$$
0 \leq F^\alpha (p) \leq |p| (c^* (p/|p|) - \alpha), \quad \text{for all } p\in \R ^N,
$$
and, as $\alpha \to 0$,
$$
F^\alpha (p)\to |p| c^* (p/|p|), \quad \text{locally uniformly in } \R  ^{N}.
$$
Let $v^\alpha_0 (x)$ be a smooth and strictly convex function such
that
$$\|\nabla v^\alpha_0 \|_\infty \leq 1, \quad v_\delta (0,\cdot) + \alpha \leq v^\alpha_0 \leq v_\delta (0, \cdot) + 2 \alpha,$$
where $v_\delta$ is the explicit viscosity solution of
$(P^0_{HJ})$ with initial data $v_\delta(0,x)$, as defined  in
Proposition \ref{prop:sol-HJ}.

Then, the solution $v^\alpha$ of the parabolic equation
\[
\quad \begin{cases}
 \partial _t v^\alpha + F^\alpha (\nabla v^\alpha) - \alpha \Delta v^\alpha=0 &\text{in }(0,\infty)\times \R ^N  \vspace{5pt}\\
 v^\alpha (0,x)= v_0^\alpha (x) &\text{in }\R ^N,
 \end{cases}
\]
is smooth, convex w.r.t. space, and converges locally uniformly to
$v_\delta$ as $\alpha \to 0$. In particular, for any $T>0$ and up
to reducing $\alpha$, the zero level set $\Gamma^\alpha_t:=\{x\in
\R^N:\, v^\alpha(t,x)=0\}$ is a smooth hypersurface for any $0
\leq t \leq T$, and is such that
\begin{equation}\label{haussdorff}
\sup_{0 \leq t \leq T} d_\mathcal{H} (\Gamma^\alpha_t , \Gamma_t) \to 0 \quad \text{ as } \alpha \to 0,
\end{equation}
where $d_\mathcal H (A,B):=\max \{ \sup_{a \in A}
dist(a,B),\,\sup_{b\in B} dist(b,A)\}$ denotes the
Hausdorff distance between two compact sets $A$ and $B$. Last,
$v^\alpha$ satisfies
\begin{equation}\label{inegalite}
\partial_t v^\alpha + |\nabla
v^\alpha| \left(c^* \left(\frac{\nabla v ^\alpha}{| \nabla v
^\alpha|}\right) - \alpha\right) \geq 0\quad \text{ in
}(0,\infty)\times \R ^N.
\end{equation}
\end{prop}
\begin{proof}  One can differentiate (in any direction)
the parabolic equation satisfied by $v^\alpha$ and, using $\|
\nabla v^\alpha_0  \|_\infty \leq 1$ for any $0<\alpha \leq \alpha
_0$, deduce from the comparison principle that
$$\|\nabla v^\alpha (t,\cdot) \|_\infty \leq 1, \quad \text{ for all } 0<\alpha \leq \alpha _0, t>0.$$
In other words, the family $(v^\alpha (t,\cdot) )_{0 < \alpha\leq
\alpha _0, t\geq 0}$ is  uniformly Lipschitz-continuous. As
confirmed by \cite{Gig2}, the proof of Theorem~4.6.3 in~\cite{Gig}
still applies thanks to the above estimate, even though the
solutions we consider are unbounded. Therefore, one can conclude
that the family of functions $v^\alpha$ converges locally
uniformly to the unique viscosity solution of $(P^0_{HJ})$ with
initial datum $v_\delta (0,x)$, namely $v_\delta$.

We now proceed by noting that, for each $0<\alpha\leq \alpha _0$,
the smoothness of $v^\alpha$ follows from standard parabolic
estimates.  One can then differentiate the parabolic equation
twice in any given direction $e\in \mathbb S ^{N-1}$ and deduce
from the comparison principle (recall that $v_0^\alpha$ is convex)
that $v^\alpha (t, \cdot)$ is convex for any positive time. In particular, we have $\Delta v^\alpha
(t,x) \geq 0$ for all $t \geq 0$ and $x\in \R^N$, which proves
\eqref{inegalite}.

Let us now turn to the convergence of the zero level set $\Gamma^\alpha_t$ of $v^\alpha$ to $\Gamma _t$. The proof again follows the steps of \cite{Gig} (see the proof of Theorem~4.6.4 in the particular case of geometric motions). We fix any
$\beta >0$ and $T>0$ and show that, for small enough $\alpha$, $\sup_{0 \leq t \leq T} d_{\mathcal{H}} (\Gamma_t^\alpha , \Gamma_t) \leq \beta$. By \eqref{inegalite}, we get that $v^\alpha (t,x) > v_\delta (t,x)$ for all $t \geq 0$ and $x \in \R^N$: in particular, $\Gamma_t^\alpha \subset \Omega_t$ for all $t \geq 0$. Let now $R>0$ be large enough so that for all $0 \leq t \leq T$ the inclusion $\Omega_t \subset B_R$ holds, where $B_R$ denotes the ball of radius $R$ and centered at the origin. By the locally uniform convergence, it is clear that for any small enough $\alpha$ and $x \in \Omega_t$ such that $d(t,x) \leq - \beta$ (recall that $d(t,\cdot)$ denotes the signed distance to $\Gamma_t$), then $v^\alpha (t,x) <0$. The convergence \eqref{haussdorff} easily follows.

Let us again fix $T >0$ and now prove that, for
small enough $\alpha$, the zero level set $\Gamma_t^\alpha$ is a
smooth hypersurface on the time interval $[0,T]$. Note that, for
any $0 \leq t_0 \leq T$ and $x_0 \in \Gamma_{t_0}^\alpha$, one has
that $|\nabla v^\alpha (t_0, x_0)| \neq 0$ provided $\alpha$ is
small enough. Otherwise, it would follow from the convexity of
$v^\alpha (t_0,\cdot)$ that $v^\alpha (t_0, \cdot) \geq 0$ in
$\R^N$, a contradiction with the fact that it approaches $v_\delta
(t_0,\cdot)$ locally uniformly. Then, as $|\nabla v^\alpha
(t_0,x_0)| \neq 0$, one can apply the implicit function theorem
and obtain the smoothness of $\Gamma_{t_0}^\alpha$.
\end{proof}

\section{Rapid emergence of the layers from
below}\label{s:generation-below}

In this section we prove that, as $\ep \to 0$, the solution
$u^\ep(t,x)$ of $(P^\ep)$ is very close to $1$ in $\Omega _0$
after a very short time. The proof relies on the spreading
properties of solutions of \eqref{monostable} with large enough
compact support at initial time~\cite{Wein02}. Precisely, the
following holds.

\begin{prop}[Emergence of the layers from below]\label{prop:generation-below} Let the nonlinearity $f$ be
of the spatially periodic monostable type, i.e. $f$ satisfies
\eqref{periodicity} and Assumption~\ref{hyp:monostable}. Let the
initial data $g$ in Problem $(P^\ep)$ satisfy Assumption~\ref{H1}.

Then, for any small $\eta >0$ and small $\alpha >0$, there is a
time $t_\alpha >0$ such that the following holds: there is $\ep
_0>0$ such that, for all $\ep\in (0,\ep_0)$,
\begin{equation}\label{estimation}
x\in \Omega _0,\; dist(x,\partial\Omega_0 ) > \alpha \
\Longrightarrow \ 1 - \eta \leq u^\ep( t_\alpha \ep ,x) \leq 1.
\end{equation}
\end{prop}

\begin{proof} Since $1$ solves the
reaction-diffusion equation in $(P^\ep)$ and since
$u^\ep(0,\cdot)=g(\cdot)\leq 1$, the comparison
principle implies $u^\ep(t,x)\leq 1$,
which  proves the upper bound in \eqref{estimation}. We next prove
the lower bound.

We begin by recalling the following result on the spreading of solutions with initial compact support
 \cite[Theorem 2.3] {Wein02}: for any $\sigma \in(0,1)$, there is $R_\sigma>0$ large enough so that the
  solution $v$ of \eqref{monostable} with initial datum $v_0 = \sigma \chi_{B_{R_\sigma}}$ converges locally
  uniformly to 1 as $t \to +\infty$. Here, $\chi$ denotes the characteristic function and $B_R$ the ball of radius
   $R$ and centered at the origin. Note that Weinberger's result~\cite{Wein02} also provides a positive spreading speed in
   any direction; however, it is not required to prove Proposition~\ref{prop:generation-below}.

Let us now fix some $\eta >0$ and $\alpha >0$. From Assumption
\ref{H1} on the initial data $g$, there is $\sigma_1\in(0,1)$ such that, for all $\ep >0$,
\begin{equation}\label{claimnew}
x\in \Omega _0,\; dist(x,\partial\Omega_0 ) >  \alpha
\ \Longrightarrow \ u^\ep(0,x)=g(x)\geq \sigma_1.
\end{equation}
We can now let $t_\alpha>0$ be such that the solution $v$ of \eqref{monostable} with initial datum $v_0 = \sigma_1 \chi_{B_{R_{\sigma_1}}}$ satisfies
\begin{equation}\label{pousse-vers-p}
v (t_\alpha , x ) \geq 1 - \eta, \quad \forall x \in B_{3R_{\sigma_1}}.
\end{equation}
We assume
without loss of generality that $R_{\sigma_1} > 2 \sqrt{N} \max _i L_i$.

Let us now fix $x^* \in \Omega _0$
such that $dist(x^*,\partial\Omega_0 ) > \alpha$. We are going to prove
\begin{equation}\label{goal}u^\ep(t_\alpha
\ep ,x^*)\geq 1 - \eta,
\end{equation}
for $\ep\in(0,\ep _0)$, where $\ep _0>0$ has to be independent on the point $x^*$ chosen as above.
We let $x_0 \in \partial \Omega_0$ such that $dist (x^* , \partial \Omega_0) = |x^*-x_0|$. Since $R_{\sigma_1} > 2 \sqrt{N} \max _i L_i$, there exists $k_\ep
^*=(k_{1,\ep}^*,..,k_{N,\ep}^*) \in \mathbb{Z}^N$ such that
\begin{equation}\label{boule}x^* - 2 R_{\sigma_1} \ep \frac{x_0 - x^*}{|x_0 - x^*|} \in\, \ep k_\ep ^*L +B_{\frac{R_{\sigma_1}\ep}{2}},
\end{equation}
where we denote $k_\ep ^*L:=(k_{1,\ep}^*L_1,..,k_{N,\ep}^*L_N)$.
Also, provided $\alpha$ and $\ep _0>0$ are small enough depending only on $0<\max _{y\in\Gamma _0} \gamma(y)<+\infty$ with $\gamma(y)$ the mean curvature (positive by convexity) of $\Gamma _0$ at point $y$,
 we have, for all $\ep\in(0,\ep _0)$,
\begin{equation}\label{eq:kZnnew}
\ep k_\ep ^*L \in \Omega_0 \; \text{ and } \; dist(\ep k_\ep^*L,
\partial \Omega_0) > \alpha +  R_{\sigma_1} \ep .
\end{equation}
Observe that if $x\notin \ep k_\ep ^* L+B_{\ep R_{\sigma_1}}$ then $v_0 \left(\frac{x-\ep k_\ep ^* L}{\ep} \right)
 = \sigma_1 \chi_{B_{R_{\sigma_1}}}\left(\frac{x-\ep k_\ep ^*L}{\ep}\right)=0$, and that if
$x\in\ep k_\ep ^* L+B_{\ep  R_{\sigma_1}}$ then \eqref{eq:kZnnew} implies that
$x\in \Omega _0$ and $dist(x,\partial\Omega_0 ) > \alpha$.
Hence, it follows from \eqref{claimnew} that
$$
g (x) \geq v \left(0,\frac{x-\ep k_\ep ^*L}{\ep}\right)
\; \text{ for all } x\in \R^N.
$$
Since $v (\frac{t}{\ep},\frac{x-\ep k_\ep^*L}{\ep})$ solves the
parabolic equation in $(P^\ep)$, the comparison principle implies
in particular that
$$
u^\ep(t_\alpha \ep,x^*)\geq
v \left(t_{\alpha},\frac{x^*-\ep k_\ep^*L}\ep\right).
$$
In view of \eqref{pousse-vers-p} and \eqref{boule}, the above
estimate implies \eqref{goal}. The
proposition is proved.
\end{proof}

The above argument also shows that, roughly speaking, the solution of $(P^\ep)$ may
only expand, which is rather natural from the dynamics of the monostable equation. Precisely the following holds.

\begin{lem}[Expansion]\label{remark42}
Let $\eta >0$ be given. Let $(\tilde \Omega _t)_{0\leq t\leq T}$ be a family of bounded and convex domains with smooth
boundaries $\tilde \Gamma _t:=\partial \tilde \Omega _t$. Then, for any $\sigma \in(0,1)$ there is a time $t_\sigma >0$
such that the following holds: there is $\ep _0>0$ --- depending only on
$0<\max _{0\leq t\leq T}\max _{y\in \tilde\Gamma _t} \gamma_t(y)<+\infty$ with $\gamma _t(y)$ the mean curvature of $\tilde \Gamma _t$
at point $y$--- such that, for any $0\leq t_0<T$, any $\ep \in(0,\ep _0)$,
$$
u^{\ep}(t_0,x) \geq \sigma,\quad \forall x\in \tilde \Omega _{t_0}\Longrightarrow
u^{\ep}(t,x) \geq 1-\eta, \quad \forall x\in \tilde \Omega _{t_0}, \forall t\geq t_0+t_\sigma \ep.
$$
\end{lem}

\section{Propagation of the layers from below}\label{s:propagation-below}

We now begin the analysis of the motion of interface. In this
section, we prove the lower estimate on the motion of level sets
of the solutions $u^\ep(t,x)$, namely statement~$(i)$ of Theorem~\ref{THEO-conv}.

To that purpose, we fix some times $0 < \tau < T $, and a small $\beta >0$. We then let $\alpha
>0$ be small enough so that the hypersurfaces $(\Gamma^\alpha_t)_{0\leq t \leq T +1}$, as defined in subsection~\ref{ss:regular}, are smooth and such that
\begin{equation}\label{etoileetoile}\sup_{0 \leq t \leq T +1} d_{\mathcal{H}} (\Gamma^\alpha_t , \Gamma_t) \leq \frac{\beta}{2}.
\end{equation}
We also denote, in this section, by $\Omega_t^{\alpha}$ the region enclosed by $\Gamma ^\alpha_t$.

\subsection{Lower estimates in small canisters}\label{ss:sub}

We start by looking, for any fixed time $t_0$, at the \lq\lq local
motion" of the interface. By \lq\lq local motion", we mean that we
will investigate the motion of the solution on small neighborhoods
of any point of $\Gamma^\alpha_{t_0}$. Precisely, the following holds.

\begin{lem}[Lower estimates in small canisters]\label{lem:pointwise}
Let $\eta >0$ be given.  Fix some time $t_0 \in (0,T)$, and assume
that
\begin{equation}\label{assumption}
x \in \Omega_{t_0}^{\alpha} \Longrightarrow u^\ep (t_0,x)
\geq 1 - \eta.
\end{equation}
Then there are two positive constants $A_1$ and $A_2$, independent
on $t_0$ and $\ep>0$ (provided it is small enough), such that
$$
u^\ep (t_0 + A_1 \sqrt{\varepsilon}, x) \geq 1 - \eta,
$$
for all $x \in D :=\cup_{x_0 \in \Gamma ^\alpha _{t_0}} \mathcal C
(x_0)$, where $\mathcal C (x_0)$ is the finite cylinder, or canister,
made of the points~$x$ such that
\begin{equation}\label{etoile}
|(x-x_0) \cdot n | \leq A_1 \left(c^* (n) - \frac{\alpha}{2}
\right) \sqrt{\varepsilon}\; \text{ and }\; |(x-x_0) \cdot
n^{\bot}| \leq A_2 \frac{\sqrt{\varepsilon}}{2}.
\end{equation}
Here $n$ denotes the unit outer normal of $\Gamma^\alpha_{t_0}$ at
point $x_0$, and $(x-x_0) \cdot n^\bot$ denotes the orthogonal
projection of $x-x_0$ on the hyperplane  $(\mathbb R n)^\bot$.
\end{lem}

\begin{proof}
First, let $\gamma>0$ be large enough so that, for all
$t\in[0,T]$, all $y\in \Gamma ^\alpha _t$ with $n_y$ the associated
unit outer normal, we have the inclusion
\begin{equation}\label{truc}
B_{\frac{1}{\gamma}} \left( y - \frac{1}{\gamma} n_y\right)
\subset \Omega_{t}^{\alpha},
\end{equation}
where $B_r(z)$ denotes the open ball of center $z$, radius $r$. By convexity, it suffices to take $\gamma$ as the
maximal curvature (in absolute value) of $\Gamma^\alpha_{t}$ in the
time interval $[0,T]$.

Let $\eta >0$ and $0<t_0<T$ be given. Let $x_0 \in \Gamma ^\alpha
_{t_0}$ be given and $n$ the associated unit outer normal. For the
lemma to be proved notice that constants $A_1$ and $A_2$, that we
need to determine, have to be independent on $t_0$, small $\ep
>0$ but also on $x_0$ and $n$. By assumption \eqref{assumption}
and inclusion \eqref{truc}, we have
$$
\forall x\in B_{\frac{1}{\gamma}} \left( x_0 - \frac{1}{\gamma}
n\right), \quad u^\ep (t_0,x) \geq 1 - \eta.
$$

We fix a constant $C>2\sqrt N \max _i L_i$ and, proceeding
similarly as in Section \ref{s:generation-below}, we can find some
point $\ep k_\ep L := \ep (k_{1,\ep} L_1, ...,k_{N,\ep} L_N)$,
where $k_{i,\ep} \in \mathbb{Z}$ for all $1 \leq i \leq N$, and
such that
\begin{equation}\label{eqn:epKL}
x_0 - \frac{n}{\gamma} \in \ep k_\ep L + B_{C \ep}.
\end{equation}
Then
\begin{equation}\label{comparison_uep}
\forall x\in B_{\frac{1}{\gamma} - C\ep} \left( \ep k_\ep
L\right), \quad u^\ep (t_0,x) \geq 1 - \eta.
\end{equation}

This leads us to study the solution $u(t,x)$ of \eqref{monostable}
with initial datum
\begin{equation}\label{blabla}
u_0 (x) := (1 -\eta) \times \chi_{B_{\frac{1}{\gamma \ep} - C}}
(x),
\end{equation}
where $B_r$ denotes the open ball centered at the origin and of
radius $r$. Note that this initial datum has compact support, so
that Theorem~\ref{th:unif_spreading} does not apply. In fact, the
solution $u (t,x)$ does not spread with speed $c^* (n)$ in the
$n$-direction as $t \to +\infty$, but rather with some minimum of
the $\frac{c^* (n')}{n \cdot n'}$ over all $n' \in
\mathbb{S}^{N-1}$. However, as the radius of the initial support
is very large, we can exhibit some transient dynamics where the
solution does spread, in any direction $n$, with speed $c^* (n)$
the minimal speed of pulsating traveling waves. Let us make this
sketch precise.

We first note that, provided that $\ep$ is small depending only on
$C$ and $\gamma$, the finite cylinder
$$
D_0:=\left\{ x\in \R^N \; : \ |  x \cdot n | \leq \frac{1}{\gamma
\ep} - 2C  \ \text{and} \ |x \cdot n^\bot | \leq
\sqrt{\frac{C}{\gamma \ep}} \right\}
$$
is a subset of $B_{\frac{1}{\gamma
\ep} - C}$ thanks to Pythagoras' theorem.  In order to
apply Theorem \ref{th:unif_spreading}, which is concerned with
planar-shaped initial data, it is more convenient to consider a
box-shaped initial support. With this in mind, we introduce $(n_1
, ... , n_{N-1})$ an orthonormalized basis of $(\mathbb R
n)^\bot$, and define the finite box
$$D_1 := \left\{ x \in \R^N \; : \ | x \cdot n | \leq \frac{1}{\gamma \ep} - 2C \
\text{and } \ \forall 1\leq i \leq N-1 , \ | x \cdot n_i | \leq
\sqrt{\frac{C}{(N-1) \gamma \ep}}  \right\},
$$
which is a subset of $D_0$.

We can now begin our investigation of the spreading of $u$, the
solution of \eqref{monostable} with initial datum \eqref{blabla}.
By the parabolic comparison principle, we have
\begin{equation*}%\label{parabolic_pointwise1}
u \geq \underline{u}\,,
\end{equation*}
where $\underline{u}$ is the solution of \eqref{monostable} with
initial datum
$$
\underline{u}_0 (x) := (1 -\eta) \times \chi_{D_1} (x).
$$

We let $\tilde{u} (t,x;n)$ denote the solution of \eqref{monostable} with
initial datum
$$
\tilde{u}_0 (x;n) := (1-\eta) \times \chi_{\{ x \cdot n \leq
\frac{1}{\gamma \ep} - 2C \} } (x),
$$
which is planar-shaped so that $\tilde{u} (t,x;n)$ spreads in the
direction $n$ with speed $c^* (n)$. Precisely, recalling $C>2\sqrt
N \max _i L_i$, we can find some point $\tilde k_\ep L := (\tilde
k_{1,\ep} L_1, ...,\tilde k_{N,\ep} L_N)$, where $\tilde k_{i,\ep}
\in \mathbb{Z}$ for all $1 \leq i \leq N$, and such that
$\frac{n}{\gamma \ep} \in B_C (\tilde k_\ep L)$. Then observe that
$$
\tilde{v}_0(x;n):=\tilde{u}_0 (x + \tilde k_\ep L ;n) \geq (1 -
\eta) \times \chi_{\{ x \cdot n \leq - 3C \} }(x).
$$
We can now apply Theorem \ref{th:unif_spreading} with the family
of functions in the right-hand side member above (which do not
depend on  $\ep$) as the family of initial data. Then, applying
the comparison principle, we get that there exists $\tau
>0$ (which does not depend on $\ep$) such that
$$
\inf_{t \geq \tau }\; \inf_{x \cdot n \leq (c^* (n)-\frac 14
\alpha) t}\; \tilde{v} (t,x;n) \geq 1 - \frac{\eta}{2},
$$
where $\tilde v (t,x;n)$ denotes the solution of \eqref{monostable} with initial datum $\tilde v _0(x;n)$. Then, since $\tilde{v}(t,x;n)=\tilde{u}(t,x+\tilde k_\ep L;n)$
thanks to the spatial periodicity, the above estimate implies
\begin{equation}\label{wein_spread0}
\inf_{t \geq \tau }\; \inf_{x \cdot n \leq \frac{1}{\gamma \ep} -
3C+(c^* (n)-\frac 14 \alpha) t}\; \tilde{u} (t,x;n) \geq 1 -
\frac{\eta}{2}.
\end{equation}
We emphasize that $\tau
>0$ can also be chosen independently of $n\in \mathbb S ^{N-1}$:
this is the exact purpose of our improvement of Weinberger's
spreading result \cite{Wein02}, namely Theorem
\ref{th:unif_spreading}.

We now estimate the difference $w := \tilde{u} - \underline{u}
\geq 0$, which satisfies $\partial_t w - \Delta w - g(t,x) w = 0$,
where
$$g(t,x):= \begin{cases}
\displaystyle \frac{f(x,\tilde{u}) - f (x,\underline{u})}{\tilde{u} - \underline{u}}  &\mbox{ if } \ w(t,x) \neq 0,\vspace{3pt}\\
\displaystyle  \partial_u f(x,\tilde{u}) & \mbox{ if } \ w(t,x)
=0.
\end{cases}
$$
From Assumption \ref{hyp:monostable}, $g(t,x)$ is uniformly
bounded by some $K$ which only depends on $f$. Then $w$ satisfies
\begin{equation}\label{eqn:linear_1}
\partial_t w - \Delta w - K w \leq 0.
\end{equation}
As this parabolic equation is linear, we infer that $w(t,x) \leq
\sum _{i=0}^{2N-2}w_i(t,x)$, where $w_0$ is the solution of
\eqref{eqn:linear_1} with initial datum
$$
w_0 (0,x)=\begin{cases}
\displaystyle 1 - \eta & \mbox{ if } x \cdot n \leq -\frac{1}{\gamma \varepsilon} +2C ,\vspace{3pt}\\
\displaystyle 0 &\mbox{ otherwise},
\end{cases}
$$
and the $w_i$'s, $1\leq i\leq N-1$, are the solutions of
\eqref{eqn:linear_1} with initial data
$$w_{2i-1} (0,x)=\begin{cases}
\displaystyle 1 - \eta & \mbox{ if } x \cdot n \leq
\frac{1}{\gamma \varepsilon}
-2C \mbox{ and } x \cdot n_i \geq \sqrt{\frac{C}{(N-1) \gamma \ep}} ,\vspace{3pt}\\
\displaystyle 0 &\mbox{ otherwise},
\end{cases}
$$
$$
w_{2i} (0,x)=\begin{cases} \displaystyle 1 - \eta & \mbox{ if } x
\cdot n \leq \frac{1}{\gamma \varepsilon}
 -2C \mbox{ and } x \cdot n_i \leq -\sqrt{\frac{C}{(N-1) \gamma \ep}}  ,\vspace{3pt}\\
\displaystyle 0 &\mbox{ otherwise}.
\end{cases}
$$
Note that, for any $e \in \mathbb{S}^{N-1}$ and any positive
constant $M$, $(t,x)\mapsto M e^{- \sqrt{K} (x \cdot e -
2\sqrt{K}t)}$ is a supersolution of the linear equation
\eqref{eqn:linear_1}. It therefore follows that
$$w_0 \left( t,x \right) \leq e^{-\sqrt{K} (x \cdot n + \frac{1}{\gamma \ep} - 2C- 2\sqrt{K}t )},$$
and, for any integer $1 \leq i \leq N-1$,
$$w_{2i-1} \left(t,x \right) \leq  e^{\sqrt{K} (x \cdot n_i  -\sqrt{\frac{C}{(N-1) \gamma \ep}}+ 2\sqrt{K}t )}, $$
$$w_{2i} \left(t,x \right) \leq  e^{-\sqrt{K} (x \cdot n_i  + \sqrt{\frac{C}{(N-1) \gamma \ep}}- 2\sqrt{K}t )}. $$
Then, we conclude that
\begin{equation}\label{eq_planarclose}
0\leq (\tilde{u} - \underline{u})\left( \frac{A_1}{\sqrt{\ep}}, x
\right) = w \left( \frac{A_1}{\sqrt{\ep}}, x \right) \leq
\sum_{i=0}^{2N-2} w_i \left( \frac{A_1}{\sqrt{\ep}} , x
\right)\leq \frac{\eta}{2},
\end{equation}
where
$$A_1:= \frac{1}{4} \sqrt{\frac{C}{K(N-1)\gamma }},$$
for all $x$ satisfying the two following inequalities:
\begin{equation}\label{ineq_11}
x \cdot n \geq -\frac{1}{\gamma \ep} + 2C + 2 A_1
\sqrt{\frac{K}{\ep} } - \frac{1}{\sqrt{K}} \ln
\left(\frac{\eta}{4N  } \right) = -\frac{1}{\gamma
\ep} + O \left( \frac{1}{\sqrt{\ep}}\right),
\end{equation}
$$|x \cdot n_i | \leq  \sqrt{\frac{C}{(N-1) \gamma \ep}} - 2 A_1 \sqrt{\frac{K}{\ep} }  + \frac{1}{\sqrt{K}}
\ln \left(\frac{\eta}{4N  } \right) = \frac{1}{2}
\sqrt{\frac{C}{(N-1) \gamma \ep}}  + O(1),$$ for $1 \leq i \leq
N-1$. The second inequality is in particular satisfied, for $\ep
>0$ small enough, if
\begin{equation}\label{ineq_12}
| x \cdot n^\bot | \leq \frac{1}{3} \sqrt{\frac{C}{(N-1)\gamma
\ep}}=:\frac{A_2}{\sqrt{\ep}}.
\end{equation}
Combining the spreading property~\eqref{wein_spread0} of
$\tilde{u}$ and inequality~\eqref{eq_planarclose}, we conclude
that
\begin{equation}\label{bidule}
\underline{u} (t_\ep, x) \geq 1 - \eta,\quad t_\ep:=\frac
{A_1}{\sqrt \ep},
\end{equation}
for any $x$ satisfying both inequalities \eqref{ineq_11} and
\eqref{ineq_12}, as well as
\begin{equation}\label{ineq_13}
x \cdot n \leq \frac{1}{\gamma \ep} - 3C + (c^* (n)-\frac 14
\alpha) t_\ep.
\end{equation}

We can now go back to our original problem $(P^\ep)$. Notice that
both $\underline u(\frac t \ep, \frac x \ep)$ and $u^\ep(t_0+t,\ep
k_\ep L+x)$ solve the equation in $(P^\ep)$. Using $D_1 \subset
B_{\frac 1{\gamma \ep}-C}$ and \eqref{comparison_uep}, we see that
$\underline u(0,\frac x\ep)\leq u^\ep(t_0,\ep k_\ep L+x)$ so that
$$
u^\ep (t_0 +t,\ep k_\ep L +x  ) \geq \underline{u} \left(\frac t
\ep, \frac x \ep\right),
$$
where $\ep k_\ep L$ satisfies \eqref{eqn:epKL}. Thus, we get
\begin{equation}\label{point}
u^\ep (t_0 + \ep t_\ep , x) \geq \underline u\left
(t_\ep,\frac{x-\ep k_\ep L}\ep\right)\geq 1 - \eta,
\end{equation}
provided that $\frac{x - \ep k_\ep L}\ep$ satisfies
\eqref{ineq_11}, \eqref{ineq_12}, \eqref{ineq_13} (so that
\eqref{bidule} holds). Now, assume that $x$ satisfies
\eqref{etoile}. Combining the first part of \eqref{etoile} and
\eqref{eqn:epKL}, we see that $\frac{x - \ep k_\ep L}\ep$
satisfies both \eqref{ineq_11} and \eqref{ineq_13}. Combining the
second part of \eqref{etoile}, $n\cdot n^\bot=0$ and
\eqref{eqn:epKL}, we see that $\frac{x - \ep k_\ep L}\ep$
satisfies \eqref{ineq_12}. Hence, \eqref{point} holds true and is
the desired conclusion that $u^\ep(t_0+A_1\sqrt \ep,x)\geq 1
-\eta$.

Note that, as announced, the constants $A_1$ and $A_2$ defined
above  depend neither on $t_0\in(0,T)$, $x_0\in \Gamma _{t_0}^{\alpha}$ and the associated  unit outer normal $n$, nor on $\ep>0$. The lemma
is proved.
\end{proof}

\begin{rem} Let us notice that Lemma \ref{lem:pointwise} shares some ideas with the so-called consistency assumption (H4) of Barles and Souganidis \cite{Bar-Sou1}. Roughly speaking, their method consists in reducing the study of the sharp interface limit to compact and smooth shapes as well as to small times, that is to consistency. In a heterogeneous and bistable context, they then proved consistency under the additional assumption that the traveling wave (which in such case is unique) depends regularly on its direction. However, such a property is far from trivial, especially in the monostable case. We therefore adopt a different approach,  relying on the uniform spreading properties proved in our earlier work \cite{A-Gil}, namely Theorem \ref{th:unif_spreading}.
\end{rem}

\subsection{Lower estimates for propagation of the layers}\label{ss:lower-estimates}

We now complete our argument by combining an iteration method and Lemma \ref{remark42}.

\begin{proof}[Proof of statement~$(i)$ of Theorem
\ref{THEO-conv}] We need to show that, for $\ep >0$ small enough, we have
$u^\ep(t,x)\geq 1 -\eta$, for all $\tau \leq t \leq T$, for all
$x$ such that $d (t,x) \leq - \beta$ (recall that $d(t,\cdot)$ denotes the signed distance function to $\Gamma _t$, negative in $\Omega _t$).

Recalling that $\Gamma _0 ^\alpha \subset \Omega _0$ and $\alpha \leq
d_{\mathcal H} (\Gamma_ 0 ^\alpha,\Gamma _0)\leq 2\alpha$ (see Proposition~\ref{regularis}), it follows from
Proposition \ref{prop:generation-below} that, for $\ep >0$ small
enough, assumption \eqref{assumption} of Lemma~\ref{lem:pointwise}
is satisfied at time $t_0 = t_\alpha \ep < \tau$. As a result
\begin{equation}\label{result}
u^\ep (t_1,x) \geq 1 - \eta, \quad t_1 := t_0 + A_1 \sqrt{ \ep},
\end{equation}
for any $x \in D$ defined as in Lemma~\ref{lem:pointwise}.
Moreover, \eqref{result}
also holds true if $x \in \Omega_{t_0}^{\alpha}$ in virtue of Lemma \ref{remark42} (notice that the needed time to
reach $1-\eta$ in Lemma \ref{remark42} is of order $\ep$), with $\Gamma _t ^\alpha$, $\Omega _t^\alpha$ playing the roles of $\tilde \Gamma _t$, $\tilde \Omega _t$. Therefore,
it follows from the claim
\begin{equation}\label{claim2}
\Omega_{t_1}^{\alpha} \subset\; D \cup
\Omega_{t_0}^{\alpha}
\end{equation}
(whose proof is postponed), that
$$
\forall x\in \Omega_{t_1}^{\alpha} \ , \quad u^\ep (t_1,x)
\geq 1 - \eta.
$$
Proceeding by induction, we conclude that for all times
$$t_k := t_0 + k A_1 \sqrt{\ep},$$
up to some $k$ such that $T < t_k < T +1$, we have
$$
u^\ep (t_k,x) \geq 1 - \eta \quad \mbox{for all} \ \  x \in
\Omega_{t_k}^{\alpha}.
$$
In particular it follows from \eqref{etoileetoile} that $u^\ep (t_k ,x) \geq 1 - \eta$ for any $x$ such that $d (t_k,x)
 \leq - \beta$.

It now only remains to consider intermediate times. Notice that, even though we stated in Lemma~\ref{remark42}
 that the solution $u^\ep$ may only expand, this is in fact only true when looking at interval of times of order
 larger than $\ep$. Therefore, the above inequality does not guarantee that $d(t,x)\leq- \beta \Rightarrow u^\ep (t,x) \geq 1 -\eta$  in
 intervals of time $[t_k,t_k+ \mathcal O(\ep)]$. To avoid this difficulty, we can nevertheless note that for all $t \in [t_k, t_{k +1})$ with $k \geq 1$,
\begin{equation}\label{presque42}
u^\ep (t,x) \geq 1 - \eta \quad \mbox{for all} \ \  x \in
\Omega_{t_{k-1}}^{\alpha}.
\end{equation}
Note that up to reducing $\ep$, we can assume that $t_1 < \tau$. Let now any $t \in [\tau, T]$, and $k \geq 1$ such
that $t \in [t_k , t_{k+1})$. Let also $x \in \Omega_t$ be such that $d(t,x) \leq -\beta$. Notice that it follows from
 Proposition~\ref{prop:motion?} that there is $C>0$ such that $d_\mathcal{H} (\Gamma_t , \Gamma_{t_{k-1}}) \leq C(t-t_{k-1})\leq 2A_1C\sqrt \ep$. Recall also that $\alpha >0$ was chosen such that
\eqref{etoileetoile} holds, so that
$d_\mathcal{H} (\Gamma_{t},\Gamma^\alpha_{t}) \leq \frac{\beta}{2}$, for all $\tau\leq t\leq T+1$.
As a result
$$
d_\mathcal{H} (\Gamma_t , \Gamma_{t^{\alpha}_{k-1}}) \leq 2A_1C\sqrt \ep+\frac \beta 2<\beta,$$
for $\ep >0$ small enough. Since $d(t,x)\leq -\beta$, this enforces $x\in \Omega^{\alpha}_{t_{k-1}}$ and, by
\eqref{presque42}, we get that $u^\ep (t,x) \geq 1 - \eta$.  This
concludes the proof of the lower estimates on the motion of the
layers of $u^\ep(t,x)$.
\end{proof}

\begin{proof}[Proof of claim \eqref{claim2}]
Recall that (see Proposition \ref{regularis}) $\Gamma^\alpha_t$ is
the zero level set of $v^\alpha(t,\cdot)$, where
\begin{equation}\label{ineg}\partial_t v^\alpha + | \nabla v^\alpha | \left( c^*
\left(\frac{\nabla v^\alpha}{ | \nabla v^\alpha |})\right) -
\alpha \right) \geq 0.
\end{equation}
To prove the claim
\eqref{claim2}, consider any $x \in \Omega_{t_1}^{\alpha}
\setminus \Omega_{t_0}^{\alpha}$, and let us prove that $x\in D$.
First, there exists some $x_0 \in \Gamma^\alpha_{t_0}$ such that
$|x - x_0| = dist (x,\Gamma^\alpha_{t_0})$ and, by convexity, such
an $x_0$ is unique. Moreover,
$$n = \frac{x-x_0}{|x-x_0|} =  \frac{ \nabla v^\alpha (t_0,x_0)}{|\nabla v^\alpha (t_0,x_0)|}$$
is, by construction, the unit outer normal of $\Gamma_{t_0}^\alpha$ at point $x_0$ (the first equality follows
from the choice of $x_0$, and the second from the definition of $\Gamma^\alpha_t$ as the zero level set of $v^\alpha (t,\cdot)$).

In order to prove that $x \in D$, it only remains to check the inequality
$$| (x - x_0) \cdot n |=| x - x_0| \leq A_1 \left(c^* (n) - \frac{\alpha}{2}\right) \sqrt{\ep}.$$
Note that, by convexity of $v^\alpha$,
$$v^\alpha (t_0, x) \geq v^\alpha (t_0,x_0) + \nabla v^\alpha (t_0,x_0) \cdot (x -x_0),$$
and also that,  thanks to the smoothness of $v^\alpha$,
$$  v^\alpha (t_1,x) - v^\alpha (t_0, x) \geq \partial_t v^\alpha (t_0 ,x) (t_1 -t_0)  - K |t_1 - t_0|^2,$$
where $K$ is a positive constant (recall that $\alpha>0$ has been
fixed). Since $x\in \Omega _{t_1}^\alpha$ we have $v^\alpha
(t_1,x)<0$ and since $x_0\in \Gamma _{t_0}^\alpha$ we have
$v^\alpha(t_0,x_0)=0$. As a result, up to increasing $K$ if
necessary,
\begin{eqnarray*}
0 &\geq& v^\alpha (t_1 ,x) - v^\alpha (t_0,x_0)\\
 & \geq &   \nabla v^\alpha (t_0,x_0) \cdot (x-x_0) + \partial_t v^\alpha (t_0 ,x) (t_1 -t_0) - K|t_1 - t_0|^2 \\
 & \geq &   \nabla v^\alpha (t_0,x_0) \cdot (x-x_0) + ( \partial_t v^\alpha (t_0 ,x_0) - K |x - x_0| )(t_1 -t_0) - K|t_1 - t_0|^2.
\end{eqnarray*}
Using \eqref{ineg}, we deduce that
$$0\geq | \nabla v^\alpha (t_0 ,x_0) | \times | x - x_0 | - | \nabla v^\alpha
(t_0,x_0) | \left( c^* (n) - \alpha \right) (t_1 - t_0)-
K[|t_1-t_0|\times |x -x_0|+ |t_1 - t_0|^2]. $$ Recalling that
$\nabla v^\alpha$ does not cancel on $\Gamma^\alpha_t$, we can
infer by compactness that
$$\rho := \inf_{0 \leq t \leq T}  \inf_{x \in \Gamma^\alpha_t} | \nabla v^\alpha (t,x) | >0.$$
Recalling also that $t_1 - t_0 = A_1 \sqrt{\ep}$, it follows from
the above that
\begin{eqnarray*}
| x -x_0 |  & \leq &  \frac{|\nabla v^\alpha (t_0,x_0)|}{|\nabla
v^\alpha (t_0,x_0)|-KA_1 \sqrt \ep}
 (c^* (n) - \alpha) A_1 \sqrt{\ep} +\frac{KA_1 ^2\ep}{\rho - K A_1 \sqrt{\ep}}\\
& \leq & \left( c^* (n) - \frac{ \alpha}{2} \right)  A_1
\sqrt{\ep},
\end{eqnarray*}
provided $\ep>0$ is small enough. As announced, $x \in D$ and the
claim is proved.
\end{proof}

\section{Control of the layers from above}\label{s:control-above}

In this section, we prove the upper estimate on the motion of
level sets of the solutions $u^\ep(t,x)$, namely statement~$(ii)$
of Theorem \ref{THEO-conv}.

To do so, we are going to  construct a family of planar
supersolutions (indexed by $y\in\Gamma _0$) for $(P^\ep)$, whose
envelop is close to the zero level sets of $v(t,\cdot)$, that is
$\Gamma _t$ in virtue of Proposition \ref{prop:motion?}. Then, for
the sake of clarity, rather than using the uniform upper spreading
speed \eqref{conclusion2}, we instead use some kind of uniform
asymptotics of the monostable minimal waves
--- which is proved in \cite{A-Gil} and actually implies
\eqref{conclusion2}.

\begin{lem}[Uniform asymptotics for critical waves, \cite{A-Gil}]\label{steep_min_wave}
Let $u^* (t,x;n) = U^* (x \cdot n - c^* (n) t , x ;n)$ be a family
of increasing in time pulsating
 traveling waves of \eqref{monostable}, with minimal speed $c^*(n)$
 in each direction $n\in \mathbb{S}^{N-1}$, shifted so that $U^* (0,0;n)= \frac{1}{2}$.

Then, the asymptotics $U^*( - \infty, x;n) = 1$, $ U^*(\infty,x;n)
= 0$ (which are uniform with respect to $x\in \R^N$) are uniform
with respect to $n \in \mathbb{S}^{N-1}$.
\end{lem}

\begin{proof}[Proof of statement $(ii)$ of Theorem~\ref{THEO-conv}]  Let $0 <  T$ and a small $\beta >0$
be given. For any $n \in \mathbb{S}^{N-1}$, denote by $U^*
(z,x;n)$ a monostable pulsating front with minimal speed $c^*(n)$
in the direction $n$, shifted so that $U^* (0,0;n)=
\frac{1}{2}$.

Thanks to  $\Vert g\Vert _\infty <1$ (see Assumption
\ref{H1}) and the above lemma, we can select some $K>0$ large
enough so that
\begin{equation}\label{shift}
U^* (z,x;n) \geq \| g\|_\infty, \quad \forall z\leq -K,\, \forall
x\in \R ^N,\, \forall n \in \mathbb S ^{N-1}.
\end{equation}
Then, for any $y \in \Gamma_0$ and denoting again by $n_y$ the
outward unit normal of $\Gamma_0$ at point $y$, we define
$$
\overline{u}(t,x):=U^* \left(\frac{(x-y).n_y-c^*(n_y)t}\ep -K,\frac x \ep
; n_y\right).
$$
From equation \eqref{eq-tw} for the traveling front, we deduce
that $\overline{u}(t,x)$ solves the parabolic equation in $(P^\ep)$. We
also have $u^\ep(0,x)=g(x)\leq \overline{u}(0,x)$: indeed, for $x\notin
\Omega _0$ we have $g(x)=0$, whereas for $x\in \Omega _0$ we have
$(x-y).n_y\leq 0$ by convexity and \eqref{shift} gives the desired
ordering. The comparison principle then implies $u^\ep(t,x)\leq
\overline{u}(t,x)$. As a result
\begin{equation}\label{controle-dessus}
0 \leq u^\ep(t,x)\leq \inf _{y\in \Gamma _0} U^*
\left(\frac{(x-y).n_y-c^*(n_y)t}\ep -K ,\frac x \ep;n_y\right).
\end{equation}

We recall that $d (t,\cdot)$ denotes the signed distance to the
set $\Gamma_t$, which is chosen to be negative in $\Omega_t$ and
positive  in $\R ^N \setminus (\Gamma _t \cup {\Omega_t})$.
 Let us
now prove that there is some $\theta >0$ such that, for any $t \in
[ 0 ,T] $ and any $x$ such that $d(t,x) \geq \beta$, then
\begin{equation}\label{eq:exist_fin}
\exists y \in \Gamma_0 , \quad (x-y).n_y - c^* (n_y) t \geq \theta \beta.
\end{equation}
Assume by contradiction  that there are some sequences $(t_k)_{k \geq 1}$,
$(x_k)_{k  \geq 1}$ as above such that
$$
\forall y \in \Gamma_0, \quad (x_k - y).n_y - c^* (n_y) t_k \leq \frac{\beta}{k}.
$$
This enforces the sequence $(x_k)_{k \geq 1}$ to be bounded so that, after
extraction of some subsequences, we are equipped with some $t_\infty
\in [0,T]$, some $x_\infty$ with $d(t_\infty,x_\infty) \geq
\beta>0$, such that
$$\forall y \in \Gamma_0, \quad (x_\infty - y).n_y - c^* (n_y) t_\infty \leq 0 .$$
Thus $v(t_\infty,x_\infty)\leq 0$, which contradicts
$d(t_\infty,x_\infty)\geq \beta$.

Let us now choose any $t \in [0, T]$, any $x$ such that $d
(t,x)\geq \beta$. In view of
 \eqref{eq:exist_fin}, we can select some $y_0 \in \Gamma_0$ such that $(x-y_0).n_{y_0} - c^* (n_{y_0}) t \geq \theta \beta.$ Then, using \eqref{controle-dessus} and the monotonicity of the
pulsating traveling wave $U^* (z,x;n)$ with respect to its first
variable, we get
$$
0\leq u^\ep(t,x)\leq U^* \left(\frac {\theta \beta}{ \ep} -K,\frac x
\ep;n_{y_0}\right)\leq \sup _{n\in \mathbb S ^{N-1}} \sup _{X \in
\R ^N} U^*\left(\frac{\theta \beta}{ \ep} -K,X ;n\right).
$$
Thanks to  Lemma~\ref{steep_min_wave}, this implies $\sup _{0
\leq t \leq T} \; \sup _{\{x: d(t,x)\geq \beta\}}\;
|u^\ep(t,x)|\to 0$ as $\ep \to 0$, which concludes the proof of
Theorem~\ref{THEO-conv}.\end{proof}
\bigskip

\noindent \textbf{Acknowledgement.} M. A. was supported by the
French {\it Agence Nationale de la Recherche} within the project
IDEE (ANR-2010-0112-01). T. G. was supported by the French {\it
Agence Nationale de la Recherche} within the project NONLOCAL
(ANR-14-CE25-0013).

We are grateful to Professor Hiroshi Matano for great hospitality
in the University of Tokyo, where this work was initiated. We
would like to thank Professor Yoshikazu Giga for enlightening
discussions around Proposition \ref{regularis}. We also thank
Professors T\'erence Bayen and Lionel Thibault for taking time to
discuss some convex analysis properties.

\end{document}